\documentclass[12pt]{amsart}

\usepackage{calligra}

\DeclareMathAlphabet{\mathcalligra}{T1}{calligra}{m}{n}
\DeclareFontShape{T1}{calligra}{m}{n}{<->s*[2.2]callig15}{}
\newcommand{\scr}{\mathcalligra{r}\,}

\usepackage{amsfonts,amssymb}
\usepackage{hyperref, graphicx}
\usepackage{epsfig}
\usepackage{latexsym}
\usepackage{amsmath,amsthm}
\usepackage{mathrsfs}
\usepackage[all,cmtip]{xy}
\theoremstyle{plain}
\newtheorem{theorem}{Theorem}[section]
\newtheorem{lemma}[theorem]{Lemma}
\newtheorem{definition}[theorem]{Definition}
\newtheorem{proposition}[theorem]{Proposition}
\newtheorem{cor}[theorem]{Corollary}
\newtheorem{remark}[theorem]{Remark}

\numberwithin{equation}{section}

\newcommand{\Nfgs}[2]{\pi^N(#1,\, #2)}
\newcommand{\Ufgs}[2]{\pi^{\mathrm{uni}}(#1,\, #2)}

\begin{document}
\baselineskip=15.5pt

\title{On an extension of Nori and local fundamental group schemes}
\author{Pavan~Adroja}
\address{Department of Mathematics, IIT Gandhinagar,
 Near Village Palaj, Gandhinagar - 382355, India}
 \email{pavan.a@iitgn.ac.in}
\author{Sanjay~Amrutiya}
\address{Department of Mathematics, IIT Gandhinagar,
 Near Village Palaj, Gandhinagar - 382355, India}
 \email{samrutiya@iitgn.ac.in}
\thanks{The research work of Pavan Adroja is financially supported by CSIR-UGC fellowship under 
the grant no. 191620092279.}
\date{}

\begin{abstract}
This article studies an extended Nori and local fundamental group schemes of Abelian varieties.
We also discuss the birational invariance of these group schemes and study their behaviour under 
the Albanese and \'etale morphisms.
\end{abstract}

\maketitle

\section{Introduction}
Madhav Nori \cite{No} has introduced the fundamental group scheme of a reduced and 
connected scheme $X$ over a field $k$. If $X$ is a proper integral scheme over a field 
$k$ with $x\in X(k)\neq \emptyset$, then the Nori fundamental group scheme, denoted by 
$\Nfgs{X}{x}$, is Tannaka dual to a neutral Tannakian category of \emph{essentially finite}
vector bundles over $X$. Recall that a vector bundle $E$ over $X$ is said to be
\emph{finite} if there are two distinct polynomials $f$ and $g$ with non-negative integer 
coefficients such that $f(E)$ is isomorphic to $g(E)$. A vector bundle $E$ over $X$ is said 
to be \emph{Nori--semistable} if for every non-constant morphism $f\,:\, C\,\longrightarrow\, X$ 
with $C$ a smooth projective curve, the pull-back $f^*E\,\longrightarrow\, C$ is semi-stable 
of degree zero. A vector bundle $E$ over $X$ is said to be \emph{essentially finite} if it 
is a Nori--semistable subquotient of a finite vector bundle over $X$. The category 
$\mathcal{C}^{\rm N}(X)$ of essentially finite vector bundles is a full subcategory of 
\textit{Nori semistable} bundles.

Adrian Langer \cite{La11, La12} studied various properties of so-called the 
$S$-fundamental group scheme. The $S$-fundamental group scheme of $(X, x)$, denoted by 
$\pi^S(X, x)$ is Tannaka dual to the category of numerically flat bundles. Recall that
a vector bundle is numerically flat if and only if it is Nori--semistable. He also gives 
the structure theorem of $\pi^S(X, x)$ for abelian varieties.

Otabe \cite{Ot} considers a slightly bigger category than the category of
finite bundles in characteristic zero. He defined the semi--finite bundle and proved 
that the category of semi-finite bundles over $X$ form the Tannakian category.  
The corresponding affine group scheme is called the extended Nori fundamental group 
scheme, and is denoted by $\pi^{\mathrm{EN}}(X,x)$. If $C$ is an elliptic curve over the field of characteristic 
zero, then $\pi^{\mathrm{EN}}(C,x)$ is the product of $\pi^{N}(C,x)$ and $\pi^{\rm uni}(C,x)$. The 
analoge of a semi-finite bundle over the field of positive characteristic is discussed in \cite{A20}. 
It is proved that for an elliptic curve $C$ over the field of positive characteristic, 
$\pi^{\mathrm{EN}}(C,x)$ is isomorphic to $\pi^{N}(C,x)$. In this paper, we generalize these two 
results to higher dimensional abelian varieties (see Theorem \ref{AVST} and Proposition \ref{ABSTP}).

If the characteristic of $k$ is zero, then the birational invariance of the $S$-fundamental group scheme 
follows from the birational invariance of the topological fundamental group. It is also known that the 
Nori fundamental group scheme is birational invariant for smooth projective varieties \cite{No}. 
If char $k = p > 0$, then the birational invariance of the $S$-fundamental group scheme is proved in 
\cite{HM}. One can use the techniques of \cite{HM} to show the birational invariance of 
$\pi^{\mathrm{EN}}(X,x)$. In fact, the proof of birational invariance of $\pi^N(X, x)$ and 
$\pi^{\rm loc}(X, x)$ is inbuilt in the proof of \cite[Theorem 1.2]{HM}. In Corollary \ref{BIP}, 
we provide additional details for completeness. We also study the behaviour of the extended Nori 
fundamental group scheme under the Albanese and \'etale morphisms (see Section \ref{sec-properties}). 
In the last section, we introduce an extension of the local fundamental group scheme and study it in 
a similar spirit.

\section{Preliminaries}
Let $X$ be a proper integral scheme defined over a perfect field $k$ endowed with a 
rational point $x\in X(k)$.

\begin{definition}
A vector bundle $V$ on $X$ is called \textit{unipotent bundle} if it has a filtration 
$$
V=V_0\supset V_1 \supset V_2 \supset ... \supset V_n=0
$$ 
such that $V_i/V_{i+1}\simeq \mathcal{O}_X$ for all $i$.
\end{definition}

The category of unipotent vector bundles $\mathcal{C}^{\rm{uni}}(X)$ over $X$ form 
the Tannakian category and the corresponding affine group scheme $\Ufgs{X}{x}$ is called 
the unipotent group scheme of $X$ \cite[Chapter IV]{No}. Let $\pi^{\rm{uni}}_{\rm ab}(X,x)$ 
be the largest abelian quotient of $\pi^{\rm{uni}}(X,x)$.

\begin{proposition}{\rm \cite[Chapter IV]{No}}
Let $X$ be a proper connected reduced scheme over the algebraically closed field 
$k$ with $x\in X(k)$. If $k$ has characteristic zero, then we have
$$\pi^{\rm{uni}}_{\rm ab}(X,x) \simeq H^1(X,\mathcal{O}_X)^*.$$
If the characteristic of $k$ is positive, then we have
$$\pi^{\rm{uni}}_{\rm ab}(X,x) \simeq \varprojlim \hat{G},$$
where $\hat{G}$ denotes the Cartier dual of $G$, and the inverse limit is taken over 
all finite local group subschemes $G$ of $\mathrm{Pic}^0\:X$.
\end{proposition}

\subsection{An extension of the Nori fundamental group scheme}
S. Otabe \cite{Ot} introduced the notion of a semi-finite bundle when the base field $k$ 
is of characteristic zero. Note that an essentially finite vector bundle is finite if 
$k$ has characteristic zero \cite[Chapter-I, \S 3]{No}.

\begin{definition}
A vector bundle $V$ on $X$ is called \textit{semi--essentially finite} bundle if it has a filtration 
$$
V=V_0\supset V_1 \supset V_2 \supset ... \supset V_n=0
$$ such that $V_i/V_{i+1}$ is an essentially finite indecomposable bundle, for all $i$.
\end{definition} 

Let $\mathcal{C}^{\mathrm{EN}}(X)$ be the Tannakian category of semi--essentially finite bundles over $X$. The 
corresponding affine group scheme is denoted by $\pi^{\mathrm{EN}}(X,x)$.

\begin{theorem}{\rm \cite[Theorem 1.1]{Ot}}
Let $C$ be an elliptic curve over an algebraically closed field $k$ of characteristic zero. 
For $k$-rational point $x\in C(k)$, we have
$$\pi^{\mathrm{EN}}(C,x)\simeq \pi^{N}(C,x)\times \pi^{\rm{uni}}(C,x).$$
\end{theorem}

\begin{theorem}{\rm \cite[Theorem 4.3]{A20}}
Let $C$ be an elliptic curve over an algebraically closed field $k$ of positive characteristic. 
For $k$-rational point $x\in C(k)$, we have
$$\pi^{\mathrm{EN}}(C,x)\simeq \pi^{N}(C,x).$$
\end{theorem}

\subsection{The $S$-fundamental group scheme}
Let $X$ be a smooth projective $k$-variety. Recall that a line bundle $L$ over $X$ is said 
to be \emph{numerically effective} if the degree of the restriction of $L$ to any irreducible 
curve $C$ in $X$ is non-negative. A vector bundle $E$ over $X$ is called numerically effective
if and only if the tautological line bundle $\mathcal{O}_{\mathbb{P}(E)}(1)$ is numerically 
effective. A vector bundle $E$ is called \emph{numerically flat} if both $E$ and its dual $E^\vee$ 
are numerically effective.

The category $\mathcal{C}^{\mathrm{nf}}(X)$ of numerically flat bundles over $X$ with a $k$-point $x$ form a neutral Tannakian 
category and the corresponding affine group scheme is called the $S$-fundamental group scheme, 
denote by $\pi^S(X, x)$ (see \cite{La11, La12} for other equivalent descriptions and various 
properties of this group scheme).

Let $G$ be any abelian group. Then we can define a diagonal group scheme $\mathrm{Diag}(G)$ as 
a functor $\mathrm{Diag}(G)\colon \mathbf{Alg}_k\longrightarrow \mathbf{Ab}$ by 
$\mathrm{Diag}(G)(R)=\mathrm{Hom}_{gr}(G, R^{\times})$. Clearly, the functor $\mathrm{Diag}(G)$ 
is represented by the $k$-group algbera $k[G]$. So, the group scheme $\mathrm{Diag}(G)$ is the 
same as taking $\mathrm{Spec}\,k[G]$ with the natural $k$-group scheme structure.

Let $\mathrm{Pic}\: X$ denote the Picard group scheme of $X$. We denote by $\mathrm{Pic}^0 X$ its 
connected component of the identity and by $\mathrm{Pic}^\tau X$ the torsion component of the identity.
Recall that the Neron--Severi group 
$\mathrm{NS}(X) = (\mathrm{Pic} X)_{\mathrm{red}}/(\mathrm{Pic}^0 X)_{\mathrm{red}}$
is finitely generated. We have a short exact sequence
$$
0 \rightarrow (\mathrm{Pic}^0 X)_{\mathrm{red}} \rightarrow (\mathrm{Pic}^\tau X)_{\mathrm{red}} 
\rightarrow \mathrm{NS}(X)_{\mathrm{tor}} \rightarrow 0,
$$
where $\mathrm{NS}(X)_{\mathrm{tor}}$ is the torsion group of $\mathrm{NS}(X)$. Therefore, we have
a short exact sequence
$$
0 \rightarrow \mathrm{Diag}(\mathrm{NS}(X)_{\mathrm{tor}})\rightarrow 
\mathrm{Diag}((\mathrm{Pic}^\tau X)_{\mathrm{red}}) \rightarrow 
\mathrm{Diag}((\mathrm{Pic}^0 X)_{\mathrm{red}}) \rightarrow 0.
$$
We refer to \cite{K05} for more details about the Picard scheme.

\begin{theorem}{\rm \cite[Theorem 5.9]{La12}}
Let $X$ be a smooth projective variety defined over an algebraically closed field $k$. 
If $\mathrm{char}\;k=0$, then we have an isomorphism
$$
\pi_{\rm ab}^S(X,x)\simeq 
H^1(X,\mathcal{O}_X)^* \times \mathrm{Diag}(\mathrm{Pic}^\tau X)\,.
$$
If $\mathrm{char}\:k=p>0$, then we have 
$$
\pi_{\rm ab}^{S}(X,x)\simeq \varprojlim \hat{G} \times \mathrm{Diag}((\mathrm{Pic}^\tau X)_{\mathrm{red}})\,,
$$
where $\hat{G}$ denotes the Cartier dual of $G$ and the inverse limit is taken over all 
finite local group subschemes $G$ of $\mathrm{Pic}^0 X$.
\end{theorem}

\noindent \textbf{Abelian varieties:}
Let $A$ be an abelian variety of dimension $g$ over an algebraically closed field $k$. 
Let $n_A:A\longrightarrow A$ be a multiplication by $n$ map and $A_n:=\text{Ker }n_A$. 
If char $k=0$, then we have $A_n \simeq (\frac{\mathbb{Z}}{n\mathbb{Z}})^{2g}$ as a group 
scheme. If char $k=p>0$, then we have (see \cite{MUM}) the following:
\begin{enumerate}
\item $A_n \simeq (\frac{\mathbb{Z}}{n\mathbb{Z}})^{2g}$ when $n$ is relatively prime to $p$;
\item $A_{p^m} \simeq A_{p^m}^{\scr} \times \hat{A}_{p^m}^{\scr} \times A_{p^m}^0$, where $A_{p^m}^{\scr}$ 
denote the reduced part of $A_{p^m}$, $\hat{A}_{p^m}^{\scr}$ denote Cartier dual of $A_{p^m}^{\scr}$ 
and $A_{p^m}^0$ is local-local part of $A_{p^m}$ for 
all $m\in \mathbb{N}$. 
Note that $A_{p^m}^{\scr}\simeq (\frac{\mathbb{Z}}{p^m\mathbb{Z}})^{r}$ and 
$\hat{A}_{p^m}^{\scr} \simeq (\mu_{p^m})^r$,  where $0\leq r \leq g$ is called the $p$-rank of $A$;
\item $A_{tp^m} \simeq A_t \times A_{p^m}$ when $t$ is relatively prime to $p$.
\end{enumerate}

The inverse limit $\varprojlim A^{\scr}_{p^n}$ is called the \textit{$p$-adic discrete Tate group} 
$T_p^d(A)$ and, the inverse limit $\varprojlim A^0_{p^n}$ is called 
\textit{$p$-adic local--local Tate group} $T_p^0(A)$.

\section{Structure theorem}
For a variety $X$ defined over a field of positive characteristic, we denote by $F_X \colon X\rightarrow X$
the Frobenius morphism. Recall that a vector bundle $V$ on $X$ is called $F$-trivial if there is a positive integer $n$ 
such that $(F_X^n)^*V\simeq \mathcal{O}_X^r$.

\subsection{The abelian part of the $\pi^{\mathrm{EN}}(X,x)$}
We first give the structure theorem for the abelian part of the group scheme $\pi^{\rm EN}(X, x)$.

\begin{lemma}\label{FL}
Let $X$ be a smooth projective variety over the algebraically closed field $k$ with $char\:k=p>0$. 
Then, every essentially finite line bundle over $X$ is a torsion (finite) line bundle.
\end{lemma}
\begin{proof}
Let $L$ be an essentially finite line bundle on $X$. By \cite[Theorem 3.5]{BH07}, there is a connected \'{e}tale Galois covering $\gamma: \hat{X}\longrightarrow X$ such that $\gamma^*L$ is 
an $F$-trivial vector bundle over $\hat{X}$, i.e., there exist $n\in \mathbb{N}$ such that 
$(F_{\hat{X}}^n)^*\gamma^*L\simeq \mathcal{O}_{\hat{X}}$. 
Since $\gamma \circ F_{\hat{X}}^n = F_{X}^n \circ \gamma$, we have
$$
(F_{\hat{X}}^n)^*\gamma^*L \simeq \gamma^*(F_{X}^n)^*L \simeq 
\gamma^*(L^{\otimes p^n}) \simeq \mathcal{O}_{\hat{X}}
$$
Therefore, $L^{\otimes p^n}$ is \'{e}tale trivializable, and hence it is a finite line bundle. 
\end{proof}

Let $\pi_{\rm ab}^{\mathrm{EN}}(X,x)$ be the largest abelian quotient of $\pi^{\mathrm{EN}}(X,x)$ which is called 
\textit{abelian part} of $\pi^{\mathrm{EN}}(X,x)$.

Let $\mathbb{F}_\tau$ denote the group of torsion (finite) line bundles on $X$. Clearly, 
$\mathbb{F}_\tau$ is a subgroup of $\mathrm{Pic}^\tau\:X$.

\begin{theorem}\label{ABEN}
Let $X$ be a smooth projective variety defined over an algebraically closed field $k$. If 
$char\:k=0$ then we have an isomorphism
$$\pi_{\rm ab}^{\mathrm{EN}}(X,x)\simeq H^1(X,\mathcal{O}_X)^* \times \mathrm{Diag}(\mathbb{F}_\tau).$$
If $char\:k=p>0$ then we have 
$$\pi_{\rm ab}^{\mathrm{EN}}(X,x)\simeq \varprojlim \hat{G} \times \mathrm{Diag}(\mathbb{F}_\tau),$$
where $\hat{G}$ denotes the Cartier dual of $G$ and the inverse limit is taken over all finite 
local group subschemes $G$ of $\mathrm{Pic}^0\:X$.
\end{theorem}
\begin{proof}
Let $\pi^{\mathrm{EN}}_{\rm{uni}}(X,x)$ be the largest unipotent quotient of $\pi^{\mathrm{EN}}(X,x)$. By the 
standard Tannakian consideration, the category of finite dimensional representation of 
$\pi^{\mathrm{EN}}_{\rm{uni}}(X,x)$ is equivalent to the full subcategory of $\mathcal{C}^{\mathrm{EN}}(X)$ 
whose objects have a filtration with all successive quotients are isomorphic to $\mathcal{O}_X$. 
Since $\mathcal{C}^{\rm{uni}}(X)\subset \mathcal{C}^{\mathrm{EN}}(X)$, we have 
$\pi^{\mathrm{EN}}_{\rm{uni}}(X,x)\simeq \pi^{\rm{uni}}(X,x)$. Therefore, the largest unipotent quotient 
of $\pi_{\rm ab}^{\mathrm{EN}}(X,x)$ which is called \textit{unipotent part} of $\pi_{\rm ab}^{\mathrm{EN}}(X,x)$ 
is isomorphic to the abelian part of $\pi^{\rm{uni}}(X,x)$.

If $\mathrm{char}\, k = 0$, then by \cite[Chapter IV, Proposition 2]{No}, the abelian 
part of $\pi^{\rm{uni}}(X,x)$ is $H^1(X,\mathcal{O}_X)^*$. If $\mathrm{char}\, k = p > 0$, 
then by \cite[Ch. IV, Proposition 6]{No}, the abelian part of $\pi^{\rm{uni}}(X,x)$ is the 
inverse limit of Cartier duals of finite local group subscheme of $\mathrm{Pic}^0\:X$.

Since the homomorphism $\pi^{\mathrm{EN}}(X,x) \longrightarrow \mathbb{G}_m$ is one dimensional 
representation of $\pi^{\mathrm{EN}}(X,x)$ and by Tannakian duality, it gives line bundle in 
$\mathcal{C}^{\mathrm{EN}}(X)$. So, the group of characters of $\pi^{\mathrm{EN}}(X,x)$ is isomorphic to 
the group of line bundles in $\mathcal{C}^{\mathrm{EN}}(X)$. By the definition of semi--essentially finite bundle 
and by Lemma \ref{FL}, we conclude that the character group of $\pi^{\mathrm{EN}}(X,x)$ is the group 
of torsion line bundles $\mathbb{F}_\tau$.
By \cite[9.5]{W}, $\pi_{\rm ab}^{\mathrm{EN}}(X,x)$ is a product of unipotent and diagonal part, and
hence the result follows.
\end{proof} 

\begin{theorem}\label{AVST}
Let $A$ be an abelian variety defined over an algebraically closed field $k$ of characteristic 
zero. Then, we have
$$\pi^{\mathrm{EN}}(A,0)\simeq \pi^{N}(A,0)\times \pi^{\rm{uni}}(A,0).$$
\end{theorem}
\begin{proof}
From the Theorem \ref{ABEN} and \cite[Theorem 6.1]{La12}, the unipotent part of $\pi^{\mathrm{EN}}(A,0)$ 
is $H^1(X,\mathcal{O}_X)^*$.

Since $\mathbb{F}_\tau=\{L\in \mathrm{Pic}^0\:A: L^{\otimes n}\simeq 
\mathcal{O}_A \text{ for some } n\in \mathbb{N}\}$, we can write $\mathbb{F}_\tau$ as a direct 
limit of $U_i:=\{L\in \mathrm{Pic}^0\:A: L^{\otimes i}\simeq \mathcal{O}_A\}$. 
$$\mathbb{F}_\tau = \varinjlim U_i$$
$$k[\mathbb{F}_\tau] = \varinjlim k[U_i]$$
$$\mathrm{Diag}(\mathbb{F}_\tau)= \text{Spec }k[\mathbb{F}_\tau] = \varprojlim \text{Spec }k[U_i]$$

Note that $U_i$ is an $i$-torsion component of $\mathrm{Pic}^0\:A$ and by \cite[\S 13]{MUM}, 
$\mathrm{Pic}^0\:A \simeq A^\vee$, where $A^\vee$ is a dual abelian variety of $A$, i.e., $U_i$ 
is an abstract group $A^\vee_i$ which is isomorphic to $(\frac{\mathbb{Z}}{i\mathbb{Z}})^{2g}$ 
as an abstract group. So, we have the following isomorphism as a group scheme. 
$$\text{Spec }k[U_i] \simeq (\mu_i)^{2g} \simeq (\frac{\mathbb{Z}}{i\mathbb{Z}})^{2g} \simeq A_i$$
$$\mathrm{Diag}(\mathbb{F}_\tau) = \varprojlim \text{Spec }k[U_i] = \varprojlim A_i.$$
Since $\pi^N(A, 0) \simeq \varprojlim A_n$ \cite[Remark 3]{No83}, we conclude the theorem.
\end{proof}

\begin{remark}\label{R1}
\rm If char $k = 0$, then by applying the same method of Theorem \ref{ABEN} for the Nori 
fundamental group scheme, and by using \cite[Proposition 2.10(1)]{Ot}, we have 
$\pi^N_{\rm ab}(X,x)\simeq \mathrm{Diag}(\mathbb{F}_\tau)$. Hence, 
$\pi^{\mathrm{EN}}_{\rm ab}(X, x) \simeq \pi^{\rm uni}_{\rm ab}(X,x)\times \pi^N_{\rm ab}(X,x)$. 
For an abelian variety $A$, the group schemes $\pi^{\mathrm{EN}}(A,0), \pi^N(A,0)$, and $\pi^{\rm uni}(A,0)$ 
being abelian, we can also conclude the Theorem \ref{AVST}.
\end{remark}

\begin{remark}\label{R2}
\rm If char $k = p > 0$, then we have $\mathcal{C}^{\rm uni}(X)\subset \mathcal{C}^N(X)$ 
\cite[Chapter IV, Proposition 3]{No}. Hence, the unipotent part of $\pi^N_{\rm ab}(X,x)$ 
is isomorphic to $\pi^{\rm uni}_{\rm ab}(X,x)$. Also, the diagonal part of $\pi^N_{\rm ab}(X,x)$ 
is isomorphic to $\mathrm{Diag}(\mathbb{F}_\tau)$. Hence, 
$\pi^{\mathrm{EN}}_{\rm ab}(X,x) \simeq \pi^N_{\rm ab}(X,x)$. For an abelian variety $A$, we get
$\pi^{\mathrm{EN}}(A,0)\simeq \pi^{N}(A,0)$. We can give more explicit proof of this as follows: 
\end{remark}

\begin{proposition}\label{ABSTP}
Let $A$ be an abelian variety defined over an algebraically closed field $k$ of positive 
characteristic, then we have 
$$\pi^{\mathrm{EN}}(A,0)\simeq \pi^{N}(A,0).$$
\end{proposition}
\begin{proof}
From the Theorem \ref{ABEN} and \cite[Theorem 6.1]{La12}, the unipotent part of $\pi^{\mathrm{EN}}(A,0)$ 
is $T_p^0(A)\times T_p^d(A)$. 

In this case, we have $U_i\simeq (\frac{\mathbb{Z}}{i\mathbb{Z}})^{2g}$, when $i$ is relatively 
prime to $p$. For any $m\in \mathbb{N}$, we have $U_{p^m}\simeq (\frac{\mathbb{Z}}{p^m\mathbb{Z}})^{r}$; 
where $0\leq r \leq g$, and $U_{tp^m}\simeq U_i\times U_{p^m}$ when $t$ is relatively prime to $p$, 
as a group. Hence, we get 
$$\mathrm{Diag}(\mathbb{F}_\tau) \simeq \varprojlim \mathcal{H}_i\,,$$
where $\mathcal{H}_i$ is $(\mu_i)^{2g}$ when $i$ is relatively prime to $p$, for any $m\in \mathbb{N}$,
we have $\mathcal{H}_{p^m}=(\mu_{p^m})^{r}$; where $0\leq r \leq g$, and $\mathcal{H}_{tp^m}= 
\mathcal{H}_t\times \mathcal{H}_{p^m}$ whenever $t$ is not divisible by $p$. Therefore,
$$\pi^{\mathrm{EN}}(A,0) \simeq T_p^0(A)\times T_p^d(A) \times \varprojlim \mathcal{H}_i$$
$$\pi^{\mathrm{EN}}(A,0) \simeq \varprojlim A^0_{p^n}\times \varprojlim A^r_{p^n} \times 
\varprojlim \mathcal{H}_i \simeq \varprojlim A_n$$

The last isomorphism follows from the decomposition of $A_n$ and some basic property of an inverse 
limit. Since $\pi^N(A, 0) \simeq \varprojlim A_n$ \cite[Remark 3]{No83}, we conclude the theorem.
\end{proof}

\begin{remark}\label{R3}
\rm If $k$ has characteristic zero, then by Theorem \ref{AVST}, it follows that for any 
semi--finite bundle $E$ on $A$, there is a unipotent bundle $U$ and a finite bundle $F$ on $A$ 
such that $E\simeq U\otimes F$. If $k$ has a positive characteristic, then by Proposition \ref{ABSTP}, 
it follows that every semi--essentially finite bundle over an abelian variety $A$ is essentially finite.
This leads to raising the following:

\noindent \textbf{Question:}
Let $X$ be a smooth projective variety defined over a field $k$ of positive characteristic. 
Does there exist a semi--essentially finite bundle over $X$ which is not an essentially finite?

If $\pi^{\acute{e}t}(X, x) = 0$, then by \cite[Corollary 8.3]{La12}, we have $\pi^S(X, x)\cong \pi^{\rm EN}(X, x) \cong \pi^N(X, x)$. 
\end{remark}

\subsection{Local unipotent group scheme}
Let $\mathcal{C}^{\rm loc}(X)$ denote the category of $F$-trivial bundle on $X$. For a rational point $x\in X(k)$, the category $\mathcal{C}^{\rm loc}(X)$ is a Tannakian category and $\pi^{\rm loc}(X,x)$ is the local fundamental group scheme corresponding to it (see \cite{MS02}).
Let us denote $\mathcal{C}^{\rm loc}_U(X)$ the intersection category of $\mathcal{C}^{\rm loc}(X)$ and $\mathcal{C}^{\rm uni}(X)$. So, this category contains vector bundles that are unipotent as well as $F$-trivial. Clearly, $\mathcal{C}^{\rm loc}_U(X)$ is a Tannakian category. Let $\pi^{\rm loc}_U(X,x)$ denote the corresponding affine group scheme and $\pi^{\rm loc}_U(X,x)_{\rm ab}$ denote the maximal abelian quotient of it.

\begin{theorem}\label{LocalU}
Let $X$ be a smooth projective variety defined over an algebraically closed field of positive characteristic. Then, 
$$\pi^{\rm loc}_U(X,x)_{\rm ab} \simeq \varprojlim G_{\ell \ell}$$
where the inverse limit is taken over all finite local-local group subschemes $G_{\ell\ell}$ of $\mathrm{Pic}^0X$.
\end{theorem}
\begin{proof}
We have $$\pi_{\rm ab}^{\rm uni}(X,x)\simeq \varprojlim \hat{G}$$
where $\hat{G}$ denotes the Cartier dual of $G$, and the inverse limit is taken over all 
finite local group subschemes $G$ of $\mathrm{Pic}^0\:X$. Since $G$ is a finite local group scheme, we can decompose it as a direct sum of local-local $G_{\ell\ell}$ and local-reduced $G_{\ell\scr}$ group scheme.
$$G\simeq G_{\ell\ell}\oplus G_{\ell\scr}$$
$$\hat{G} \simeq \hat{G_{\ell\ell}}\oplus \hat{G_{\ell\scr}}$$

Recall that a local-local group scheme means the group scheme and its Cartier dual are both local group schemes. Therefore, the dual of local-local group scheme $\hat{G_{\ell\ell}}$ is also a local-local group scheme. The local-reduced group scheme means the group scheme is local, and its Cartier dual is a reduced group scheme. Therefore, the dual of local-reduced scheme $\hat{G_{\ell\scr}}$ is reduced (more specifically reduced-local). 

By the universal property of $\pi^{\rm loc}(X,x)$, we have $\pi^{\rm loc}(X,x)$ as the inverse limit of the system of a finite local group scheme, which decomposes into the local-local and local-reduced part. Since $\mathcal{C}^{\rm loc}_U(X)=\mathcal{C}^{\rm loc}(X) \cap \mathcal{C}^{\rm uni}(X)$, we need to find the common component of group scheme in $\pi_{\rm ab}^{\rm uni}(X,x)$ and $\pi^{\rm loc}(X,x)$, which is a local-local group scheme as mentioned above. Hence,
$$\pi^{\rm loc}_U(X,x)_{\rm ab} \simeq \varprojlim G_{\ell\ell}$$
where the inverse limit is taken over all finite local-local group subschemes $G_{\ell\ell}$ of $\mathrm{Pic}^0X$.
\end{proof}

\begin{remark}\label{rem-Eloc}
\rm Let $C$ be an ordinary elliptic curve over an algebraically closed field of characteristic $p>0$. 
Let $\mathcal{I}_{(r,0)}$ denote the isomorphism classes of rank $r$ and degree $0$ indecomposable 
vector bundle on $C$. From Atiyah's classification of vector bundles over an elliptic curve \cite{At}, 
there is rank $r$ vector bundle $E_r\in \mathcal{I}_{(r,0)}$ such that all other vector bundles in 
$\mathcal{I}_{(r,0)}$ is of the form $E_r\otimes L$ for some $L\in \mathrm{Pic}^0C$ also there is 
an exact sequence
$$
0\rightarrow E_{r-1} \rightarrow E_r \rightarrow \mathcal{O}_C \rightarrow 0
$$
Hence, $E_r$ is an unipotent vector bundle. From \cite[Proposition 2.10]{Od}, we have 
$(F^n_C)^*E_r\simeq E_r$ for all $n\in \mathbb{N}$ and hence $E_r$ cannot be $F$-trivial bundle.
\end{remark}

\begin{theorem}\label{ABLocal}
Let $A$ be an abelian variety of $p$-rank $r$ defined over an algebraically closed field $k$ of characteristic $p>0$ then,
we have
$$\pi^{\rm loc}(A,0)\simeq T_p^0(A) \times \varprojlim (\mu_{p^n})^r$$
\end{theorem}
\begin{proof}
\textbf{Diagonal part:} As we know, the group of characters is the same as the group of one-dimensional representations and by Tannaka duality, which is the same as the group of line bundles in $\mathcal{C}^{\rm loc}(A)$. Let $\mathbb{F}_p$ denote the group of line bundle in $\mathcal{C}^{\rm loc}(A)$.
$$\mathbb{F}_p = \{L\in \mathcal{C}^{\rm loc}(A) : (F^n)^*L\simeq \mathcal{O}_A \text{ for some } n\in \mathbb{N}\}$$
$$\mathbb{F}_p = \{L\in \mathcal{C}^{\rm loc}(A) : L^{\otimes p^n}\simeq \mathcal{O}_A \text{ where } n\in \mathbb{N}\} $$
$$\mathbb{F}_p \simeq \varinjlim \left( \frac{\mathbb{Z}}{p^n\mathbb{Z}}\right)^r $$
$$\mathrm{Spec}\:k[\mathbb{F}_p] \simeq  \varprojlim (\mu_{p^n})^r$$
\textbf{Unipotent part:} Unipotent part is the affine group scheme corresponding to the Tannakian category $\mathcal{C}^{\rm loc}(A) \cap \mathcal{C}^{\rm uni}(A)$, that is $\pi^{\rm loc}_U(A,0)$. By Theorem \ref{LocalU}, we have 
$$\pi^{\rm loc}_U(A,0) \simeq \varprojlim G_{\ell\ell} \simeq T_p^0(A)$$
where the inverse limit is taken over all finite local-local group subschemes $G_{\ell\ell}$ of $\mathrm{Pic}^0A$.
Hence, by \cite[9.5]{W}, the result follows.
\end{proof}

\begin{remark}\label{ABX}
\rm From the above Theorem \ref{ABLocal}, one can observe that the maximal abelian quotient of local fundamental group scheme denoted by $\pi^{\rm loc}_{\rm ab}(X,x)$ is isomorphic to the product of $\pi^{\rm loc}_U(X,x)_{\rm ab}$ and $\mathrm{Diag}(\mathbb{F}_p)$, where $\mathbb{F}_p$ denote the abstract group of $p^n$-torsion line bundles for all $n\in \mathbb{N}$, which is subgroup of $\mathrm{Pic}^\tau X$.
\end{remark}

\section{Properties}\label{sec-properties}
\subsection{The Albanese morphism}
Let $X$ be a smooth projective variety over an algebraically closed field $k$ and $x\in X$ be 
a fixed point. The variety $\mathrm{Alb}\:X$ is dual to the reduced scheme underlying 
$\mathrm{Pic}^0\:X$. Note that $\mathrm{Alb}\:X$ is an abelian variety and let 
$alb_X:X\longrightarrow \mathrm{Alb}\:X$ be the Albanese morphism mapping $x$ to 
$0_{\mathrm{Alb}\:X}$. 
So, there is an induced homomorphism $\psi: \pi^{\mathrm{EN}}_{\rm ab}(X,x)\longrightarrow 
\pi^{\mathrm{EN}}(\mathrm{Alb}\:X,0_{\mathrm{Alb}\:X})$. In this section, we aim to find the kernel 
of the homomorphism $\psi$. The kernel of multiplication by $n$ map $n_S\colon S\longrightarrow S$ is denoted by $_nS$ for a commutative group scheme $S$.

If the characteristic of $k$ is positive, then we have $\pi^{\mathrm{EN}}_{\rm ab}(X,x)\simeq \pi^N_{\rm ab}(X,x)$ 
(see Remark \ref{R2}). By \cite[Corollary 7.2]{La12}, we conclude the same result for the 
fundamental group scheme $\pi^{\mathrm{EN}}(X, x)$. 

\begin{theorem}
Let $X$ be a smooth projective variety over an algebraically closed field $k$ of the characteristic zero. 
Then, we have the following short exact sequence:
$$
0\rightarrow \mathrm{Diag}(\mathrm{NS}(X)_{\mathrm{tor}})\rightarrow 
\pi^{\mathrm{EN}}_{\rm ab}(X,x) \rightarrow \pi^{\mathrm{EN}}(\mathrm{Alb}\:X,0_{\mathrm{Alb}\:X}) \rightarrow 0
$$ 
\end{theorem}
\begin{proof}
Let $\mathbb{F}_\tau(X)$ and $\mathbb{F}_\tau(\mathrm{Alb}\:X)$ denote the group of finite line bundles on $X$ and $\mathrm{Alb}\:X$ respectively. 
By Theorem \ref{ABEN}, we have the following
$$\pi^{\mathrm{EN}}_{\rm ab}(X,x)\simeq H^1(X,\mathcal{O}_X)^* \times \mathrm{Diag}(\mathbb{F}_\tau(X)),$$
and by replacing $X$ with $\mathrm{Alb}\:X$, we have
$$
\pi^{\mathrm{EN}}(\mathrm{Alb}\:X,0_{\mathrm{Alb}\:X})\simeq 
H^1(\mathrm{Alb}\:X,\mathcal{O}_{\mathrm{Alb}\:X})^* \times \mathrm{Diag}(\mathbb{F}_\tau(\mathrm{Alb}\:X))\,.
$$
Since $\mathrm{Alb}\:X$ is an abelian variety over the field of characteristic zero, we have 
$\mathrm{Pic}^\tau\:(\mathrm{Alb}\:X) \simeq \mathrm{Pic}^0\:(\mathrm{Alb}\:X)$ and
$\mathrm{Pic}^0(\mathrm{Alb}\:X)\simeq (\mathrm{Pic}^0\:X)_{\rm red} \simeq \mathrm{Pic}^0\:X$.

Let $T_0(\mathrm{Pic}^0\:X)$ denote the tangent space at $0$. We know that 
$T_0(\mathrm{Pic}^0\:X)\simeq H^1(X,\mathcal{O}_X)$. On the other hand, we have 
$T_0(\mathrm{Pic}^0\:(\mathrm{Alb}\:X)) \simeq H^1(\mathrm{Alb}\:X, \mathcal{O}_{\mathrm{Alb}\:X})$. 
As $\mathrm{Pic}^0\:(\mathrm{Alb}\:X) \simeq \mathrm{Pic}^0\:X$, we get 
$H^1(X,\mathcal{O}_X) \simeq H^1(\mathrm{Alb}\:X, \mathcal{O}_{\mathrm{Alb}\:X})$. Hence, 
the unipotent part of $\mathrm{Ker}\:\psi$ is trivial.

As $\mathbb{F}_\tau(X)$ denote the torsion part of the group $\mathrm{Pic}^{\tau}\:X$, so we 
can write $$\mathbb{F}_\tau(X) = \varinjlim_{n\in \mathbb{N}}\: _n\mathrm{Pic}^{\tau}\:X\,$$
and
$$
\mathbb{F}_\tau(\mathrm{Alb}\:X) = \varinjlim_{n\in \mathbb{N}}\: 
_n\mathrm{Pic}^{\tau}\:(\mathrm{Alb}\:X) =  \varinjlim_{n\in \mathbb{N}}\: _n\mathrm{Pic}^{0}\:X.
$$
Since
$$
0\rightarrow \mathrm{Pic}^0\:X \rightarrow \mathrm{Pic}^{\tau}\:X \rightarrow 
\mathrm{NS}(X)_{\mathrm{tor}}\rightarrow 0
$$
is an exact sequence, we get an exact sequence
$$
0\rightarrow\: _n\mathrm{Pic}^0\:X \rightarrow\: _n\mathrm{Pic}^{\tau}\:X 
\rightarrow\: _n\mathrm{NS}(X)_{\mathrm{tor}}\rightarrow 0
$$
for all $n\in \mathbb{N}$.

Since $\mathrm{NS}(X)_{\mathrm{tor}}$ is finite, there is some $N\in \mathbb{N}$ such that 
$_N\mathrm{NS}(X)_{\mathrm{tor}}=\mathrm{NS}(X)_{\mathrm{tor}}$ and by taking direct limit, we get
$$
0\rightarrow \mathbb{F}_\tau(\mathrm{Alb}\:X) \rightarrow \mathbb{F}_\tau(X) 
\rightarrow \mathrm{NS}(X)_{\mathrm{tor}}\rightarrow 0\,.
$$
So, we have an exact sequence
$$
0\rightarrow \mathrm{Diag}(\mathrm{NS}(X)_{\mathrm{tor}})
\rightarrow \mathrm{Diag}(\mathbb{F}_\tau(X))\rightarrow 
\mathrm{Diag}(\mathbb{F}_\tau(\mathrm{Alb}\:X))\rightarrow 0\,.
$$
Therefore, the diagonal part of $\mathrm{Ker}\:\psi$ is 
$\mathrm{Diag}(\mathrm{NS}(X)_{\mathrm{tor}})$.
\end{proof}

\begin{proposition}\label{BUALB}
Let $X$ be a smooth projective variety over an algebraically closed field $k$ of the characteristic $p>0$. Then, we have the following short exact sequence:
$$0\rightarrow \varprojlim_{G \subset \mathrm{Pic}^0X} G/G_{\rm red} \times \mathrm{Diag}\left(\varinjlim_{n\in \mathbb{N}}\: _{p^n}\mathrm{NS}(X)_{\mathrm{tor}}\right)\rightarrow \pi^{\rm loc}_{\rm ab}(X,x) \rightarrow \pi^{\rm loc}(\mathrm{Alb}\:X,0_{\mathrm{Alb}\:X}) \rightarrow 0$$
where the inverse limit is taken over all the local-local group subscheme $G$ of $\mathrm{Pic}^0X$.
\end{proposition}
\begin{proof}
The Albenese morphism $alb_X: X \longrightarrow \mathrm{Alb}\:X$ induced a morphism between the local fundamental group scheme ${\varphi}: \pi^{\rm loc}_{\rm ab}(X,x) \longrightarrow \pi^{\rm loc}(\mathrm{Alb}\:X,0_{\mathrm{Alb}\:X})$. By Remark \ref{ABX}, we have 
$$\pi^{\rm loc}_{\rm ab}(X,x)\simeq \varprojlim_{G_{\ell\ell}\subset \mathrm{Pic}^0X} G_{\ell\ell} \times \mathrm{Diag}(\mathbb{F}_p)$$
where the inverse limit is taken over all the local-local group subschemes $G_{\ell\ell}$ of $\mathrm{Pic}^0X$ and $\mathbb{F}_p$ is the subgroup of $(\mathrm{Pic}^\tau X)_{\rm red}$.

Since $\mathrm{Pic}^\tau (\mathrm{Alb}\:X)\simeq \mathrm{Pic}^0(\mathrm{Alb}\:X)\simeq (\mathrm{Pic}^0X)_{\rm red}$, we have
$$\pi^{\rm loc}(\mathrm{Alb}\:X,0_{\mathrm{Alb}\:X}) \simeq \varprojlim_{(G_{\ell\ell})_{\rm red}\subset (\mathrm{Pic}^0X)_{\rm red}} (G_{\ell\ell})_{\rm red} \times \mathrm{Diag}(\mathbb{F}_p(\mathrm{Alb}\:X))$$
where the inverse limit is taken over all the local-local group subschemes $(G_{\ell\ell})_{\rm red}$ of $(\mathrm{Pic}^0X)_{\rm red}$ and $\mathbb{F}_p(\mathrm{Alb}\:X)$ is the subgroup of $\mathrm{Pic}^\tau (\mathrm{Alb}\:X)\simeq (\mathrm{Pic}^0X)_{\rm red}$.

We have the following exact sequence:
$$
0\rightarrow (G_{\ell\ell})_{\rm red} \rightarrow G_{\ell\ell} \rightarrow G_{\ell\ell}/(G_{\ell\ell})_{\rm red} \rightarrow 0\,.
$$
By taking dual and passing through the inverse limit, we get the unipotent part of the kernel.

Note that $\mathbb{F}_p(X)$ denote the group of $p^n$ torsion line bundles for all $n\in \mathbb{N}$. We can write 
$$\mathbb{F}_p(X) = \varinjlim_{n\in \mathbb{N}}\: _{p^n}(\mathrm{Pic}^{\tau}X)_{\rm red}$$
Similarly, we have 
$$\mathbb{F}_p(\mathrm{Alb}\: X) = \varinjlim_{n\in \mathbb{N}}\: _{p^n}(\mathrm{Pic}^{0}X)_{\rm red}$$

For all $n\in \mathbb{N}$, we have the following exact sequence: 
$$0 \rightarrow\: _{p^n}(\mathrm{Pic}^0X)_{\rm red} \rightarrow\: _{p^n}(\mathrm{Pic}^{\tau}X)_{\rm red} \rightarrow\: _{p^n}\mathrm{NS}(X)_{\rm tor} \rightarrow 0$$
By passing through the direct limit, we get the diagonal part of the kernel.
\end{proof}

\subsection{Birational invariance}

\begin{proposition}\label{BU}
Let $Y$ be a smooth complete variety over an algebraically closed field $k$ and let $f \colon X\longrightarrow Y$ be the blow-up of $Y$ along a smooth subvariety $Z\subset Y$. Then $\pi^{\mathrm{EN}}(X,x)\longrightarrow \pi^{\mathrm{EN}}(Y,f(x))$ is an isomorphism.
\end{proposition}
\begin{proof}
Let $E\in \mathcal{C}^{\mathrm{EN}}(X)$. Then by \cite[Lemma 8.3]{La11}, we have $f_*E$ is locally free and $E\simeq f^*f_*E$ as $\mathcal{C}^{\mathrm{EN}}(X)$ is the full subcategory of $\mathcal{C}^{nf}(X)$. By \cite[Lemma 8.3]{La11} and $\pi^{\mathrm{EN}}(X,x)$ is a quotient of $\pi^{S}(X,x)$, we have the following commutative diagram in which the horizontal maps are surjective and the above vertical map is an isomorphism. So, the below map $\pi^{\mathrm{EN}}(X,x)\longrightarrow \pi^{\mathrm{EN}}(Y,f(x))$ is faithfully flat.
\[
\xymatrix{    
    \pi^S(X,x) \ar[r]^\simeq \ar[d]_{p} & \pi^S(Y,f(x)) \ar[d]^p \\ 
    \pi^{\mathrm{EN}}(X,x) \ar[r]_{}       & \pi^{\mathrm{EN}}(Y,f(x))    } 
\]
To see that $\pi^{\mathrm{EN}}(X,x)\longrightarrow \pi^{\mathrm{EN}}(Y,f(x))$ is a closed immersion, by 
\cite[Proposition 2.21(b)]{DM}, it is enough to prove that for every object 
$E \in \mathcal{C}^{\mathrm{EN}}(X)$, there is an object $E' \in \mathcal{C}^{\mathrm{EN}}(Y)$ such that 
$E\simeq f^*E'$. 

Suppose $E$ is an essentially finite bundle on $X$. By \cite[Theorem 1.1]{BS} there is a 
proper $k$-scheme $P$ with a proper surjective morphism $p:P\longrightarrow X$ such that 
$p^*E$ is trivial. 

Since $(f\circ p)^*(f_*E)\simeq p^*f^*f_*E \simeq p^*E$ is trivial and $f\circ p$ is proper 
morphism, it follows that $f_*E$ is an essentially finite bundle.
If $E$ is semi--essentially finite bundle on $X$, then by induction on rank, we can conclude 
that $f_*E$ is semi--essentially finite bundle. To see this, let $E$ be a semi--essentially 
finite bundle of rank $r$ on $X$, i.e., there is a filtration
$$
E=E_0\supset E_1\supset E_2 \supset...\supset E_{n-1}\supset E_n=0
$$ 
such that $E_i/E_{i+1}$ is an indecomposable essentially finite bundle.
From the filtration, $E_{n-1}$ is an essentially finite subbundle of $E$. Then, the bundle $f_*E_{n-1}$ 
is an essentially finite. Let $F=E/E_{n-1}$. Note that $F$ is a semi--essentially finite bundle 
and $\mathrm{rank}\: F < \mathrm{rank}\: E$ so, by induction, $f_*F$ is semi--essentially finite 
bundle. We have the following exact sequences:
\begin{align}
0\longrightarrow E_{n-1} \longrightarrow E \longrightarrow F \longrightarrow 0 \label{A}
\end{align}
\begin{align}
0\longrightarrow f_*E_{n-1} \longrightarrow f_*E \longrightarrow f_*F \label{B}
\end{align}
For any numerically flat bundle $V$, we have $f^*f_*V \simeq V$ and hence 
$\mathrm{rank } \: f_*V = \mathrm{rank } \: V$. It follows that the last map of 
(\ref{B}) is surjective. Hence, $f_*E$ is a semi--essentially finite bundle.
\end{proof}

By the weak factorization theorem \cite[12.4]{WJ} and Proposition \ref{BU}, we have the 
following:
\begin{cor}\label{BI0}
Let $X$ and $Y$ be smooth projective varieties over an algebraically closed field $k$ 
of characteristic zero and $\phi:X\dashrightarrow Y$ be a birational map. 
Let $x_0\in X(k)$ be a point where $\phi$ is defined. Then $\pi^{\mathrm{EN}}(X,x_0)$ is isomorphic 
to $\pi^{\mathrm{EN}}(Y,\phi(x_0))$.
\end{cor}

\begin{cor}\label{BIP}
Let $X$ and $Y$ be smooth projective varieties over an algebraically closed field $k$ of positive characteristic and $\phi \colon X\dashrightarrow Y$ be a birational map. Let $x_0\in X(k)$ be a point where $\phi$ is defined. Then, we have $\pi^{EN}(X,x_0)\simeq \pi^{EN}(Y,\phi(x_0))$ and $\pi^{\rm loc}(X,x_0) \simeq \pi^{\rm loc}(Y,\phi(x_0))$.
\end{cor}
\begin{proof}
If $\dim X \leq 2$, then the result follows from \cite[Chapter-V, Proposition 5.3]{Ha} and Proposition \ref{BU}.
Assume that $\dim X \geq 3$.
Consider the diagram
$$
    \xymatrix{    
   & Z \ar[dl]_{p} \ar[dr]^{q} & \\
    X \ar@{.>}[rr]_{\phi}  &     & Y   } 
$$
where $Z$ is the normalization of a closure of the graph of $\phi$, the morphisms $p$ and $q$ are birational. In view of \cite[Theorem 1.2]{HM}, we need to show that $q_*p^*E$ is a semi--essentially finite whenever $E$ is a semi--essentially finite bundle on $X$.

Since $q_*p^*$ commutes with the tensor product, it follows that $q_*p^*E$ is a finite bundle whenever $E$ is a finite bundle on $X$. Let $E$ be an essentially finite bundle on $X$. Then, there is $V_1,V_2\in \mathcal{C}^{\rm nf}(X)$ such that $V_2$ is subbundle of $V_1$, $V_1$ is subbundle of finite bundle and $E\simeq V_1/V_2$. We have the following exact sequence:
$$0\rightarrow V_2 \rightarrow V_1 \rightarrow V_1/V_2 \rightarrow 0$$
$$0\rightarrow q_*p^*V_2 \rightarrow q_*p^*V_1 \rightarrow q_*p^*(V_1/V_2) \rightarrow 0$$

First note that for any $V\in \mathcal{C}^{\rm nf}(X)$, the rank of $q_*p^*V$ is the same as that of $V$. To see this,
let us adopt the setting mentioned in proof of \cite[Theorem 1.2]{HM}. Let $S$ be a general smooth surface in $Y$. Since 
$q\colon Z\rightarrow Y$ is birational, there exists a smooth projective surface $T$ and a proper birational morphism 
$\pi \colon T\longrightarrow S$ with the following commutative diagram:
\begin{equation}\label{eq-diag}
    \xymatrix{    
   T \ar[r]^{h} \ar[d]_{\pi}& Z \ar[d]^{q} \\
   S \ar[r]_{i}  & Y } 
\end{equation}
Let $W=p^*V$. Then, the morphism $i^*q_*W\rightarrow \pi_*h^*W$ is an isomorphism. Hence, we get
$$
(i\circ \pi)^*q_*W = \pi^*i^*q_*W\simeq \pi^*\pi_*h^*W \simeq h^*W
$$ in which last isomorphism follows from \cite[Lemma 8.3]{La11}. By the commutativity of diagram \ref{eq-diag}, 
we have $h^*W\simeq (i\circ \pi)^*q_*W \simeq (q\circ h)^*q_*W = h^*q^*q_*W$. Therefore, the rank of $W$ is the same 
as the rank of $q_*W$.

Hence, $q_*p^*E \simeq q_*p^*(V_1/V_2) \simeq (q_*p^*V_1)/(q_*p^*V_2)$. This proves that $q_*p^*E$ is an essentially finite bundle on $Y$.

By the similar arguments as in the proof of Proposition \ref{BU}, using induction on rank, we can prove $q_*p^*E$ is a semi--essentially finite bundle whenever $E$ is a semi--essentially finite bundle. 

To see that $\pi^{\rm loc}(X,x_0) \simeq \pi^{\rm loc}(Y,\phi(x_0))$, it is enough to prove $q_*p^*V$ is $F$-trivial bundle on $Y$ whenever $V$ is $F$-trivial on $X$.
Let $V$ be $F$-trivial bundle on $X$. Then, $W:=p^*V$ is $F$-trivial on $Z$. i.e. $(F_Z^n)^*(W)\simeq \mathcal{O}_Z^{\oplus r}$ for some $n,r\in \mathbb{N}$. From \cite[Lemma 3.1]{HM}, we have $(F_X^n)^*(q_*W)\simeq q_*((F_Z^n)^*(W)) \simeq q_*(\mathcal{O}_Z^{\oplus r}) \simeq (q_*\mathcal{O}_Z)^{\oplus r} = (\mathcal{O}_Y)^{\oplus r}$. Hence, $q_*p^*V$ is $F$-trivial on $Y$.
\end{proof}

\subsection{Behaviour under \'{e}tale morphism}
Let $\psi: X\longrightarrow Y$ be an \'{e}tale Galois cover with a finite Galois group $G$. 
Let $x_0\in X$ and $y_0\in Y$ be $k$-points such that $f(x_0)=y_0$. 
The pullback by $\psi$ induces a map $\psi^*:\mathcal{C}^{\mathrm{EN}}(Y)\longrightarrow \mathcal{C}^{\mathrm{EN}}(X)$ 
which is a morphism of Tannakian categories. So, this induces a morphism
$\tilde{\psi}: \pi^{\mathrm{EN}}(X,x_0)\longrightarrow \pi^{\mathrm{EN}}(Y,y_0)$.
Consider the following diagram:
\[
\xymatrix@C=2.5pc{
**[l]X\times G \ar[dr]_{p_1}
\ar@/^20pt/[rr]^\mu \ar[r]^-\simeq & X\times_Y X
\ar[d]^-p \ar[r]_-p & X \ar[d]^\psi \\
& X \ar[r]_\psi & Y
}
\]
For any $E\in \mathcal{C}^{\mathrm{EN}}(X)$, we have the followings:
$$
\psi^*\psi_*E \simeq p_*p^*E \simeq p_{1_*}\mu^*E \simeq \bigoplus_{g\in G}g^*E 
$$

\begin{proposition}\label{Et}
The map $\tilde{\psi}: \pi^{\mathrm{EN}}(X,x_0)\longrightarrow \pi^{\mathrm{EN}}(Y,y_0)$ is a closed immersion.
\end{proposition}
\begin{proof}
If $E$ is an essentially finite bundle on $X$, then there exist a connected \'{e}tale Galois cover 
$\gamma: \hat{X}\longrightarrow X$ such that $\gamma^*E$ is $F$-trivial bundle (if char $k = 0$, 
then $\gamma^*E$ is trivial). Note that $\psi\circ\gamma:\hat{X}\longrightarrow Y$ 
\'{e}tale covering of $Y$. Since
$$
(\psi\circ\gamma)^*(\psi_*E) \simeq \gamma^*\psi^*\psi_*E \simeq 
\gamma^*(\bigoplus_{g\in G}E)\simeq \bigoplus_{g\in G} \gamma^*g^*E\;,
$$
it follows that $(\psi\circ\gamma)^*(\psi_*E)$ is $F$-trivial (or trivial, if char $k = 0$).
Hence, the bundle $\psi_*E$ is an essentially finite.

Let $E$ be a semi--essentially finite bundle of rank $r$ on $X$, i.e., there is a filtration
$$
E=E_0\supset E_1\supset E_2 \supset...\supset E_{n-1}\supset E_n=0
$$ 
such that $E_i/E_{i+1}$ is an indecomposable essentially finite bundle.
From the filtration, $E_{n-1}$ is essentially finite subbundle of $E$ and $E/E_{n-1}$ is 
semi--essentially finite bundle on $X$. Then, $\psi_*E_{n-1}$ is essentially finite and by 
induction, we have $\psi_*(E/E_{n-1})$ is semi--essentially finite bundle.

Let $r(V)$ denote the rank of vector bundle $V$. Consider the following exact sequence:
\begin{align}
0 \longrightarrow \psi_*E_{n-1} \longrightarrow \psi_*E \longrightarrow \psi_*(E/E_{n-1}) \label{A2}
\end{align}
We have 
$$
r(\psi_*(E/E_{n-1})) = |G|\cdot r(E/E_{n-1}) = |G|\cdot (r(E)-r(E_{n-1})) = r(\psi_*E)-r(\psi_*E_{n-1}).
$$
Hence, the last map of (\ref{A2}) is 
surjective. Since the category of semi--essentially finite bundle is closed under taking 
extensions, it follows that $\psi_*E$ is a semi--essentially finite bundle.

We need to prove that any $E\in \mathcal{C}^{\mathrm{EN}}(X)$ is subquotient of the $\psi^*E'$ for 
some $E'\in \mathcal{C}^{\mathrm{EN}}(Y)$. Let $E'=\psi_*E$ then $\psi^*\psi_*E \simeq \bigoplus_{g\in G} g^*E$. 
Clearly, $E$ is subquotient of $\bigoplus_{g\in G} g^*E$. Hence, $\tilde{\psi}$ is a closed immersion. 
\end{proof}

By using the criteria \cite[Theorem A.1]{ES} and Proposition \ref{Et}, we have:
\begin{cor}
If $\psi: X\longrightarrow Y$ is an \'{e}tale Galois cover with a finite Galois group $G$, 
then we get the following exact sequence:
$$
0\rightarrow \pi^{\mathrm{EN}}(X,x_0) \rightarrow \pi^{\mathrm{EN}}(Y, y_0) \rightarrow G \rightarrow 0.
$$
\end{cor}

\section{An extension of the local fundamental group scheme}
Let $X$ be a smooth projective variety over the algebraically closed field of characteristic $p>0$. 
Let $F_X \colon X\longrightarrow X$ be a Frobenius morphism.

\begin{definition}
A vector bundle $V$ on $X$ is called $F$-semi trivial if there is a filtration
$$V=V_0\supset V_1\supset ...\supset V_{n-1}\supset V_n=0$$
such that $V_i/V_{i+1}$ is indecomposable $F$-trivial bundle on $X$ for all $i\in\{0,...,n-1\}$.
\end{definition}

Let $\mathcal{C}^{\rm loc}_E(X)$ denote the category of $F$-semi trivial bundles on $X$. Note that simple objects of this category are indecomposable $F$-trivial bundles on $X$, and this category is closed under an extension. By following arguments of \cite[Proposition 2.14]{Ot}, one can prove that the category $\mathcal{C}^{\rm loc}_E(X)$ is $k$-linear abelian rigid tensor category. We give details for completion in the following Proposition \ref{Eloc-T}.

\begin{proposition}{\rm(cf. \cite[Proposition 2.14]{Ot})}\label{Eloc-T}
The category $\mathcal{C}^{\rm loc}_E(X)$ is $k$-linear abelian rigid tensor category.
\end{proposition}
\begin{proof}
First, let us prove $\mathcal{C}^{\rm loc}_E(X)$ is closed under the tensor product and dual. If $E$ is $F$-trivial and $F$ is $F$-semi trivial then we have the following filtration
$$F\otimes E=F_0\otimes E\supset F_1\otimes E\supset ... \supset F_{n-1}\otimes E\supset F_n\otimes E=0$$
such that $(F_i\otimes E)/(F_{i+1}\otimes E) \simeq (F_i/F_{i+1})\otimes E$ is $F$-trivial for all $i$. So, $F\otimes E$ is $F$-semi trivial. Now, let $E$ and $F$ are $F$-semi trivial bundles. We have the filtration
$$E=E_0\supset E_1\supset ... \supset E_{m-1}\supset E_m=0$$
such that $E_i/E_{i+1}$ is indecomposable $F$-trivial for all $i$. After tensoring with $F$ we have,
$$E\otimes F=E_0\otimes F\supset E_1\otimes F\supset ... \supset E_{m-1}\otimes F\supset E_m\otimes F=0$$
such that $(E_i\otimes F)/(E_{i+1}\otimes F)\simeq E_i/E_{i+1}\otimes F$ is $F$-semi trivial.

We have the following exact sequence:
$$0\rightarrow E_{i+1}\otimes F \rightarrow E_{i}\otimes F \rightarrow E_{i}/E_{i+1}\otimes F \rightarrow 0$$
As the category $\mathcal{C}^{\rm loc}_E(X)$ closed under extensions and by induction, we have $E\otimes F$ is $F$-semi trivial. To prove $E^\vee$ is $F$-semi trivial, let us consider $(E^\vee)_i = (E/E_{m-i})^\vee$.
$$
E_{m-i}\subset E_{m-i-1} \Rightarrow E/E_{m-i}\supset E/E_{m-i-1} 
$$
$$
\Rightarrow (E/E_{m-i})^\vee \supset (E/E_{m-i-1})^\vee \Rightarrow (E^\vee)_{i} \supset (E^\vee)_{i+1}
$$

This gives the filtration 
$$E^{\vee}=(E^\vee)_{0} \supset(E^\vee)_{1} \supset... \supset(E^\vee)_{m-1} \supset (E^\vee)_{m}=0$$
such that $(E^\vee)_{i}/(E^\vee)_{i+1} \simeq (E_{m-i-1}/E_{m-i})^\vee$ which is $F$-trivial for all $i$.
Similarly, by following the arguments as in the proof of \cite[Proposition 2.14]{Ot}, it can be checked 
that $\mathcal{C}^{\rm loc}_E(X)$ is an abelian category.
\end{proof}

For a rational point $x\in X(k)$, the category $\mathcal{C}^{\rm loc}_E(X)$ with the fiber functor 
$x^*:\mathcal{C}^{\rm loc}_E(X)\longrightarrow \mathbf{Vect}_k$ is a neutral Tannakian category. 
Let $\pi^{\rm loc}_E(X,x)$ denote the corresponding affine group scheme. In general, the group scheme $\pi^{\rm loc}_E(X,x)$ is bigger than $\pi^{\rm loc}(X,x)$ (see Remark \ref{rem-Eloc}).

\begin{theorem}\label{ABELocal}
Let $X$ be a smooth projective variety defined over an algebraically closed field $k$ of characteristic $p>0$. Let $x\in X(k)$ be a rational point then,
we have
$$\pi^{\rm loc}_E(X,x)_{\rm ab}\simeq \pi^{\mathrm{uni}}_{\rm ab}(X,x) \times \mathrm{Diag}(\mathbb{F}_p)$$
\end{theorem}
\begin{proof}
Note that $\mathcal{C}^{\rm uni}(X)$ is the full subcategory of $\mathcal{C}^{\rm loc}_E(X)$. So, unipotent part of $\pi^{\rm loc}_E(X,x)_{\rm ab}$ is $\pi^{\mathrm{uni}}_{\rm ab}(X,x)$.

For a diagonal part, a group of line bundles in $\mathcal{C}^{\rm loc}_E(X)$ is same as the group $\mathbb{F}_p$ of line bundles in $\mathcal{C}^{\rm loc}(X)$. Hence, the theorem follows from the decomposition of a commutative group scheme \cite[9.5]{W}.
\end{proof}

\begin{cor}
Let $A$ be an abelian variety of $p$-rank $r$ defined over an algebraically closed field $k$ of characteristic $p>0$ then,
we have
$$\pi^{\rm loc}_E(A,0)\simeq \pi^{\mathrm{uni}}(A,0) \times \varprojlim_{n\in \mathbb{N}} (\mu_{p^n})^r$$
\end{cor}
\begin{proof}
The result follows using Theorem \ref{ABELocal} and Theorem \ref{ABLocal}.
\end{proof}

\begin{cor}
We have the following exact sequence:
$$0\rightarrow \varprojlim_{G \subset \mathrm{Pic}^0X} \widehat{G/G_{red}} \times \mathrm{Diag}\left(\varinjlim_{n\in \mathbb{N}}\: _{p^n}\mathrm{NS}(X)_{\mathrm{tor}}\right)\rightarrow \pi^{\rm loc}_{E}(X,x)_{\rm ab} \rightarrow \pi^{\rm loc}_E(\mathrm{Alb}\:X,0_{\mathrm{Alb}\:X}) \rightarrow 0$$
where the inverse limit is taken over all the local group subschemes $G$ of $\mathrm{Pic}^0X$.
\end{cor}
\begin{proof}
We know that the unipotent part of $\pi^{\rm loc}_{E}(X,x)_{\rm ab}$ is same as unipotent part of $\pi^{N}_{\rm ab}(X,x)$ and the diagonal part of $\pi^{\rm loc}_{E}(X,x)_{\rm ab}$ is same as diagonal part of $\pi^{\rm loc}(X,x)_{\rm ab}$. Hence, the lemma follows from \cite[Corollary 7.2]{La12} and Proposition \ref{BUALB}.
\end{proof}

\begin{remark}
\rm
\begin{enumerate}
\item If $\pi^{\text{\'et}}(X, x)=0$, then all horizontal maps in the following diagram 
$$
    \xymatrix{    
    \pi^{\mathrm{EN}}(X,x) \ar[r] & \pi^{N}(X,x) \ar[dr] \ar[r] & \pi^{\rm loc}_E(X,x) \ar[d] \ar[r] & \pi^{\rm loc}(X,x) \ar[d] \\
     &  & \pi^{\mathrm{uni}}(X,x) \ar[r] & \pi^{\rm loc}_U(X,x) \\
   } 
$$ 
are isomorphisms \cite[Proposition 8.2]{La12}.
\item The birational invariance property of the extended local fundamental group scheme follows from the birational invariance property of the local fundamental group scheme (Corollary \ref{BIP}) and by the same argument mentioned in Proposition \ref{BU}.
\end{enumerate}

\end{remark}



\end{document}